\documentclass[12pt]{amsart}

\usepackage{amsthm}

\usepackage{amssymb}
\usepackage{verbatim}
\usepackage[curve,matrix,arrow]{xy}
\usepackage{amsmath,amsfonts}
\usepackage{graphicx}
\usepackage{epsf}

\usepackage{amssymb, graphics}

\newcommand{\mathsym}[1]{{}}

\newcommand{\thmref}[1]{Theorem~\ref{#1}}

\newcommand{\lemref}[1]{Lemma~\ref{#1}}
\newcommand{\eqnref}[1]{Equation~(\ref{#1})}

\newcommand{\corref}[1]{Corollary~\ref{#1}}

\newcommand{\figref}[1]{Figure~\ref{#1}}

  {\end{list}}

\newtheorem{theorem}{Theorem}[section]

\newtheorem{corollary}[theorem]{Corollary}

\newtheorem{lemma}[theorem]{Lemma}

\theoremstyle{definition}

\newcommand{\ee}[1]{E(#1)}
\newcommand{\vv}[1]{V(#1)}

\def\<{\langle }
\def\>{\rangle }
\newcommand{\secref}[1]{\S\ref{#1}}

\begin{document}

\title[The Number of Spanning Trees]
{The Number of Spanning Trees for The Generalized Cones of $K_n$, The Generalized Half Cones of $K_{m,n}$ and Some Family of Modified $K_{m,n}$}

\author{Zubeyir Cinkir}
\address{Zubeyir Cinkir\\
Department of Industrial Engineering\\
Abdullah Gul University\\
38100, Kayseri, TURKEY\\}
\email{zubeyir.cinkir@agu.edu.tr}



\begin{abstract}
We compute the total number of spanning trees for the generalized cone of the complete graph $K_n$ and a number of families of some modified bipartite graphs $K_{m,n}$. In particular, we obtain a new method of finding the number of spanning trees of $K_n$ and $K_{m,n}$. Our method relies on the vertex deletion formula for the number of spanning trees.
\end{abstract}

\maketitle

\section{Introduction}\label{sec introduction}

Let $G$ be a connected graph possibly having self-loops and multiple edges. The graph
$G$ has the set of vertices $\vv{G}$ and the set of edges $\ee{G}$. We denote the total number of spanning trees of $G$ by $t(G)$.

For the complete graph $K_n$, it is known that $t(K_n)=n^{n-2}$, \cite{Ca} and \cite{KDT}.
We also know that $t(K_{n_1,n_2, \ldots, n_k})=n^{k-2} \prod_{i=1}^k (n-n_i)^{n_i-1}$ 
for any $k$-partite graph $K_{n_1,n_2, \ldots, n_k}$, where
where $n=n_1+n_2+\cdots + n_k$ \cite{Au}, \cite{ER}, \cite{HB} and \cite{L}. In particular, $t(K_{m,n})=n^{m-1} m^{n-1}$.

Given a graph $G$, by adding a vertex $p$ to $V(G)$ and by adding $m\geq 1$ multiple edges between $p$ and each vertices of $G$ we obtain the generalized cone of $G$. We denote this graph by $C^mG$. 
For example, if $G=K_n$ the complete graph on $n$ vertices, the graph on the left in \figref{fig3} illustrates $C^3K_3$. In this case, $C^1K_n$ is the cone of $K_n$, which is nothing but $K_{n+1}$. In \secref{sec complete}, we found that
$$t(C^m K_n)=m (m+n)^{n-1}.$$

Let $K_{m,n}$ be the complete bipartite graph with the vertex set $V_1 \cup V_2$, where $V_1=\{ p_1, \, \ldots, \, p_n \}$ and $V_2=\{ q_1, \, \ldots, \, q_m \}$. From this graph, we obtain the graph $M^kK_{m,n}$ by replacing each edge between the vertex $q_m$ and the vertices $p_i$ with $k$ multiple edges. We call this graph modified bipartite graph.
The graph on the right in \figref{fig3} illustrates the case with $k=2$, $m=3$ and $n=4$. When $k=1$, $M^kK_{m,n}$ is the usual bipartite graph $K_{m,n}$. In \secref{sec Bipartite}, we found that
$$t(M^k K_{m,n})=k \cdot n^{m-1} (m+k-1)^{n-1}.$$

\begin{figure}
\centering
\includegraphics[scale=0.6]{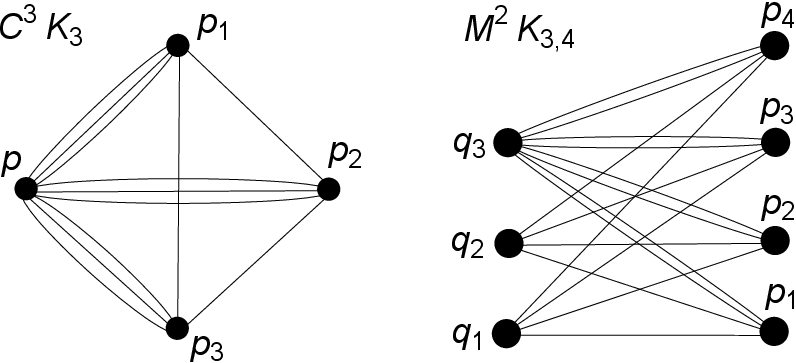} \caption{The graphs  $C^3K_3$ and $M^2K_{3,4}$.} \label{fig3}
\end{figure}

Given $K_{m,n}$ with the vertex set $V_1 \cup V_2$, where $V_1$ and $V_2$ are as given above, we obtain the graph $M^{k_1,k_2,\ldots, k_m} K_{m,n}$ by replacing each edge between the vertices $q_i$ and $p_j$ by $k_i$ multiple edges. In this way, both graphs have the same vertex set, but $M^{k_1,k_2,\ldots, k_m} K_{m,n}$ has $n(k_1+\ldots+k_m)$ edges while $K_{m,n}$ has $n m$ edges. We call  $M^{k_1,k_2,\ldots, k_m} K_{m,n}$ the generalized complete bipartite graph. Note that $M^{1,1,\ldots, 1} K_{m,n}=K_{m,n}$. The graph on the left in \figref{figmkandfmk} illustrates the case with $k_1=3$, $k_2=2$, $m=2$ and $n=3$. In \secref{sec bipartite general}, we showed that
$$
t(M^{k_1,k_2,\ldots, k_m} K_{m,n})= n^{m-1} k_1k_2\cdots k_m (k_1+k_2+\cdots+k_m)^{n-1}.
$$
\begin{figure}
\centering
\includegraphics[scale=0.6]{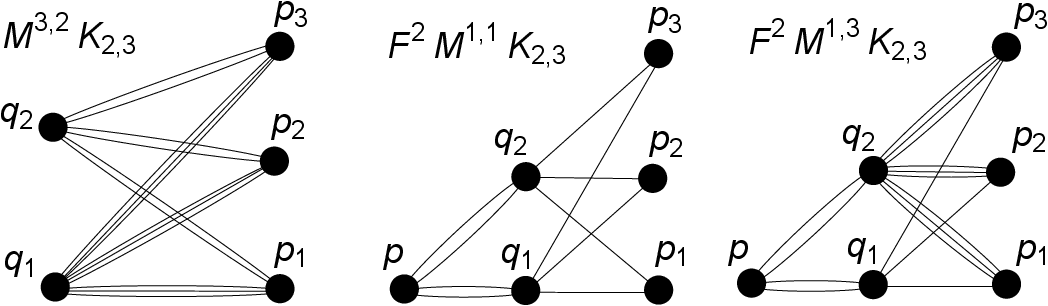} \caption{The graphs $M^{3,2}K_{2,3}$, $F^2M^{1,1}K_{2,3}$ and $F^2M^{1,3}K_{2,3}$.} \label{figmkandfmk}
\end{figure}

We obtain the graph $F^kM^{k_1,k_2,\ldots, k_m} K_{m,n}$ from the graph $M^{k_1,k_2,\ldots, k_m} K_{m,n}$
by adding a vertex $p$ and adding $k$ edges between $p$ and the vertex $q_j$ for each $j=1,\,2,\ldots,m$.
We call this graph the generalized half cone of $M^{k_1,k_2,\ldots, k_m} K_{m,n}$. The graph in the middle in \figref{figmkandfmk} illustrates the case with $k=2$, $k_1=k_2=1$, $m=2$ and $n=3$. The graph on the right in \figref{figmkandfmk} illustrates the case with $k=2$, $k_1=1$, $k_2=3$, $m=2$ and $n=3$. In \secref{sec half cone bipartite}, we showed that
$$
t(F^kM^{k_1,k_2,\ldots, k_m} K_{m,n})= (k_1+k_2+\cdots+k_m)^{n-1} k \Big( \prod_{i=1}^m (k+k_i n) \Big) \sum_{i=1}^m \frac{k_i}{k+k_i n}.
$$

For any two vertices $p, \, q \, \in \vv{G}$,  $G_{pq}$ is the graph obtained from $G$ by identifying these vertices.

The following theorem shows how $t(G)$ behaves under the deletion of a vertex of $G$:
\begin{theorem}\cite[Theorem 5.8]{C1}\label{thm main}
Let $u \in V(G)$, $N_G(u)= \{ p_1, \, \ldots, \, p_n \} \subset V(G)$ for a graph $G$, and let $u$ be adjacent to the vertex $p_i$ via by $a_i \geq 1$ number of edges for each $i \in \{1, \, \ldots, \, n \}$ with $n \geq 2$. If $u$ is not a cut vertex, then for $G$ and $H=G-u$ we have
\begin{equation*}
\begin{split}
t(G)&= \Big( \sum_{i=1}^n a_i \Big) t(H) + \sum_{\substack{S \subset N_G(u) \\ |S| \geq 2}} \Big( \prod_{i \in I_S } a_i \Big) t(H_{S}),
\end{split}
\end{equation*}
where $I_S$ is the set of indexes of the vertices in $S$, and  $H_S$ is the graph obtained from $H$ by identifying all vertices in $S$. 
\end{theorem}

Our method for finding the number of spanning trees
relies on \thmref{thm main}.

\section{Generalized Cones of Complete Graphs}\label{sec complete}

In this section, we compute the number of spanning trees of $C^m K_n$, the generalized cone of complete graph $ K_n$.
\begin{theorem}\label{thm Kn}
For any integers $m \geq 1$ and $n \geq 0$, we have
\begin{equation*}\label{}
\begin{split}
t(C^m K_n)=m (m+n)^{n-1}.
\end{split}
\end{equation*}
\end{theorem}
\begin{proof}
We obtain the proof by strong induction on $n$.

\textbf{Case $n=0$:} Since $C^mK_0$ is the graph with one vertex and $m$ cycles, $t(C^m K_0)=1$. This agrees with the fact that $m(m+0)^{0-1}=1$.

\textbf{Case $n=1$:} In this case,  $C^mK_1$ is the graph with two vertices and $m$ multiple edges. This is also known as Banana graph $B_m$. Therefore, $t(C^m K_1)=m$. This agrees with the fact that $m(m+1)^{1-1}=m$.

Suppose $t(C^m K_s)=m (m+s)^{s-1}$ for integers $s$ such that $0 \leq s \leq n-1$, where $n \geq 1$ is a given integer.

Let $V=\{p_1, p_2, \ldots, p_n \}$ be the set of vertices of the complete graph $K_n$ so that $C^m K_n$ has the vertex set $V \cup \{ p \}$. Note that $C^mK_n-p$ is nothing but $K_{n}$. We set $H:=K_{n}$.

Since the neighbourhood of the vertex $p$ is $N_{C^m K_n}(p)=V$ and that $p$ is connected to $p_i$ via $m$ edges for each index $i$, applying \thmref{thm main} gives
\begin{equation*}\label{eqn Kn1}
\begin{split}
t(C^mK_n) &= m \cdot t(H)+ \sum_{S \subset V, \, \,  |S| \geq 2} m^{|S|} t(H_{S})\\
&=m \cdot t(H)+  \sum_{j=2}^n \sum_{S \subset V, \, \, |S| =j} m^j t(H_{p_1 p_2 \cdots p_j}), \quad \text{by the symmetries in $H$}\\
&= m \cdot t(H)+  \sum_{j=2}^n \binom{n}{j} m^j t(H_{p_1 p_2 \cdots p_j})
\end{split}
\end{equation*} 
We note that $H_{p_1 p_2 \cdots p_j}$ is the same as the graph $C^jK_{n-j}$ with different vertex labeling, and that $H=K_n=C^1K_{n-1}$. Thus, we have 
\begin{equation*}\label{eqn Kn2}
\begin{split}
t(C^mK_n) &=\sum_{j=1}^n \binom{n}{j} m^j t(C^jK_{n-j})\\
&=\sum_{j=1}^n \binom{n}{j} m^j j n^{n-j-1}, \quad \text{by the induction assumption}\\
&=\sum_{j=1}^n \binom{n-1}{j-1} m^j n^{n-j}, \quad \text{since $\binom{n}{j}=\frac{n}{j}\binom{n-1}{j-1}$}\\
&=\sum_{j=0}^{n-1} \binom{n-1}{j} m^{j+1} n^{n-1-j}, \quad \text{by the change of index}\\
&=m (m+n)^{n-1}, \quad \text{by Binomial Theorem}.
\end{split}
\end{equation*} 
This completes the proof.
\end{proof}

Since $K_n=C^1K_{n-1}$, we obtain the following immediate consequence of \thmref{thm Kn}:
\begin{corollary}\label{cor Kn}
\begin{equation*}\label{}
\begin{split}
t(K_n)=n^{n-2}.
\end{split}
\end{equation*}
\end{corollary}
\corref{cor Kn} is known as Cayley's Theorem.

\section{Modified Complete Bipartite Graphs}\label{sec Bipartite}

Let $K_{m,n}$ be given with the vertex set $V_1 \cup V_2$, where 
$V_1=\{ p_1, \, \ldots, \, p_n \}$ and $V_2=\{ q_1, \, \ldots, \, q_m \}$. 
Let the modified complete bipartite graph $M^k K_{m,n}$ be as defined in \secref{sec introduction}.
\begin{theorem}\label{thm MKmn}
For any integers $m \geq 1$, $n \geq 1$ and $k \geq 1$, we have
\begin{equation*}\label{}
\begin{split}
t(M^k K_{m,n})=k \cdot n^{m-1} (m+k-1)^{n-1}.
\end{split}
\end{equation*}
\end{theorem}
\begin{proof}
We first note that $M^kK_{m,n}-q_m=M^1K_{m-1,n}=K_{m-1,n}$. We set $H:=K_{m-1,n}$.
For $N=m+n$ with $m \geq 1$ and $n \geq 1$, we give the proof by the strong induction on $N$.

\textbf{Case $n=1$:} In this case, $M^k K_{m,1}$ is a graph obtained from the banana graph $B_k$ by attaching a star graph $S_{m-1}$ to one of its vertices. Thus,  $t(M^k K_{m,1})=k$ for any $k \geq 1$ and $m \geq 1$. On the other hand, $k \cdot 1^{m-1} (m+k-1)^{1-1}=k$.

\textbf{Case $m=1$:} The graph $M^k K_{1,n}$ is the one point union of $n$ copies of $B_k$. Then by the multiplicative property of the number of spanning trees,  $t(M^k K_{1,n})=k^n$ for each $k \geq 1$ and $n \geq 1$. This agrees with the number $k \cdot n^{1-1} (1+k-1)^{n-1}=k^n$.

Suppose the given formula holds for each integer $m$ and $n$ with $s=m+n$ such that $2 \leq s <N$, where $N$ is some integer. Then our aim is to show that it also holds for $N$.

Since $N_{M^k K_{m,n}}(q_m)=V_1$ and the vertex $q_m$ is connected to $p_i$ via $k$ edges
for each $i=1, \ldots, n$, we use \thmref{thm main} to obtain
\begin{equation*}\label{eqn MKmn1}
\begin{split}
t(M^kK_{m,n}) &= n \cdot k \cdot t(H)+ \sum_{S \subset V_1, \, \,  |S| \geq 2} k^{|S|} t(H_{S})\\
&=n \cdot k \cdot t(H)+  \sum_{j=2}^{n} \sum_{S \subset V_1, \, \, |S| =j} k^j t(H_{p_1 p_2 \cdots p_j}), \quad \text{by the symmetries in $H$}\\
&= n \cdot k \cdot t(H)+  \sum_{j=2}^n \binom{n}{j} k^j t(H_{p_1 p_2 \cdots p_j})
\end{split}
\end{equation*} 
We have $t(M^1K_{m-1,n})=t(K_{m-1,n})=t(K_{n,m-1})=t(M^1K_{n,m-1})$, and  $t(H_{p_1 p_2 \cdots p_j})=t(M^jK_{n-j+1,m-1})$. Thus,
\begin{equation*}\label{eqn MKmn2}
\begin{split}
t(M^kK_{m,n}) &= n \cdot k \cdot t(H)+  \sum_{j=2}^n \binom{n}{j} k^j  t(H_{p_1 p_2 \cdots p_j})\\
&= n \cdot k \cdot t(M^1K_{n,m-1})+  \sum_{j=2}^n \binom{n}{j} k^j t(M^jK_{n-j+1,m-1})\\
&=\sum_{j=1}^n \binom{n}{j} k^j t(M^jK_{n-j+1,m-1})\\
&=\sum_{j=1}^n \binom{n}{j} k^j j (m-1)^{n-j} n^{m-2}, \quad \text{by the induction assumption}\\
&=\sum_{j=1}^n \binom{n-1}{j-1} k^j (m-1)^{n-j} n^{m-1}, \quad \text{since $\binom{n}{j}=\frac{n}{j}\binom{n-1}{j-1}$}\\
&=n^{m-1}k\sum_{j=0}^{n-1} \binom{n-1}{j} k^j (m-1)^{n-1-j}, \quad \text{by the change of index}\\
&=k n^{m-1} (m+k-1)^{n-1}, \quad \text{by Binomial Theorem}.
\end{split}
\end{equation*}
This is what we want to show.
\end{proof}

Since $K_{m,n}=M^1K_{m,n}$, we obtain the following immediate consequence of \thmref{thm MKmn}:
\begin{corollary}\label{cor Kmn}
\begin{equation*}\label{}
\begin{split}
t(K_{m,n})=n^{m-1}m^{n-1}.
\end{split}
\end{equation*}
\end{corollary}

\section{Generalized Complete Bipartite Graphs}\label{sec bipartite general}

In this section, we consider a generalization of $M^kK_{m,n}$. Namely, the generalized complete bipartite graph 
$M^{k_1,k_2,\ldots, k_m} K_{m,n}$. This graph is a multiple edge version of $K_{m,n}$. Let $V_1 \cup V_2$ be the vertex set of $M^{k_1,k_2,\ldots, k_m} K_{m,n}$, where  $V_1$ and $V_2$ are as defined in 
\secref{sec introduction}. 

We first need a preliminary lemma.
\begin{lemma}\label{lem counting}
Let $A=\{ k_1, \, k_2, \, \ldots, \, k_m  \}$. For any integer $1 \leq j \leq m$, we have
\begin{equation*}
\begin{split}
\sum_{B \subset A, \, \,  |B| = j} \, \,  \sum_{k \in B } k =\binom{m-1}{j-1}(k_1+k_2+\cdots+k_m).
\end{split}
\end{equation*}
\end{lemma}
\begin{proof}
In this summation, each $k_i$ appears  $\binom{m-1}{j-1}$ times, which is the number of $j$ element subsets of $A$ that contain $k_i$. Then the result follows.
\end{proof}

Next, we give a formula for the total number of spanning trees of the generalized complete bipartite graph $M^{k_1,k_2,\ldots, k_m} K_{m,n}$:
\begin{theorem}\label{thm gen MKmn}
For any integers $m \geq 1$, $n \geq 1$ and $k_i \geq 1$ for each $i= 1, \, 2, \ldots, m$, we have
\begin{equation*}\label{}
\begin{split}
t(M^{k_1,k_2,\ldots, k_m} K_{m,n})= n^{m-1} k_1k_2\cdots k_m (k_1+k_2+\cdots+k_m)^{n-1}.
\end{split}
\end{equation*}
\end{theorem}
\begin{proof}
We first note that $M^{k_1,k_2,\ldots, k_m} K_{m,n}-p_n=M^{k_1,k_2,\ldots, k_m} K_{m,n-1}$. We denote this graph by $H$. For the sake of brevity, we set $T:=k_1+k_2+\cdots+k_m$ and $P:= k_1 k_2 \cdots k_m$.

When $S= \{ q'_{1}, \, q'_{2}, \, \ldots, \, q'_{j} \} \subset V_2=\{ q_1, \, \ldots, \, q_m \}$, we have the complement set $V_2-S= \{ q'_{j+1}, \, q'_{j+2}, \, \ldots, \, q'_{m} \}$. For such a set $S$, we obtain the graph $H_S$ by identifying the $j$ vertices in $S$. Suppose that the number of edges between the vertices 
$q'_i$ and $p_j$ is $k'_i$. Then we note that 
$H_S$ is the same as the graph $M^{k'_{1}+ k'_{2}+ \cdots + k'_{j}, \, k'_{j+1},k'_{j+2},\ldots, k'_m} K_{m-j+1,n-1}$ with possibly different vertex labeling.

We prove the given formula by strong induction on $n$.

\textbf{Case $n=1$:} In this case, $M^{k_1,k_2,\ldots, k_m} K_{m,1}$ is the one point union of $m$ banana graphs $B_{k_i}$, where $i= 1, \, 2, \ldots, m$. Thus,  $t(M^{k_1,k_2,\ldots, k_m} K_{m,1})=k_1k_2\cdots k_m$ for any $m \geq 1$. This agrees with the given formula for $n=1$.

Suppose that  the formula in the theorem holds for each integers $s$ such that $1 \leq s < n$, where $n$ is some integer. Our aim is to show that it also holds for $n$. 
First, note that
\begin{equation}\label{eqn gen MKmn1}
\begin{split}
H_S=(n-1)^{m-j} \Big( \sum_{i \in I_S } k'_i \Big) \Big( \prod_{i \in I_{V_2-S} } k'_i \Big) T^{n-2}
\end{split}
\end{equation}
for any set $S \subset V_2$ with $|S| \geq 2$ by the assumption we made. Here, $I_S$ is the set of indexes of the vertices in $S$. Similarly,
\begin{equation}\label{eqn gen MKmn2}
\begin{split}
t(H)=(n-1)^{m-1} k_1 k_2 \cdots k_m T^{n-2}.
\end{split}
\end{equation}
Since $N_{M^{k_1,k_2,\ldots, k_m} K_{m,n}}(p_n)=V_2$, \thmref{thm main} gives
\begin{equation*}\label{}
\begin{split}
t(M^{k_1,k_2,\ldots, k_m} K_{m,n}) &= T \cdot t(H)+ \sum_{S \subset V_2, \, \,  |S| \geq 2} 
\Big( \prod_{i \in I_S } k'_i \Big)  t(H_{S})\\
&= T \cdot t(H)+ \sum_{j=2}^m\sum_{S \subset V_2, \, \,  |S| \geq j} 
\Big( \prod_{i \in I_S } k'_i \Big)  t(H_{S})
\end{split}
\end{equation*}
We continue by using \eqnref{eqn gen MKmn1},
\begin{equation*}\label{}
\begin{split}
&=T \cdot t(H)+
  \sum_{j=2}^m \sum_{S \subset V_2, \, \,  |S| \geq j} \Big( \prod_{i \in I_S } k'_i \Big) 
 (n-1)^{m-j} \Big( \prod_{i \in I_{V_2-S} } k'_i \Big) \Big( \sum_{i \in I_S } k'_i \Big)  T^{n-2}\\
&=T \cdot t(H)+
  \sum_{j=2}^m \sum_{S \subset V_2, \, \,  |S| \geq j} 
 (n-1)^{m-j} \Big( \sum_{i \in I_S } k'_i \Big)  P \cdot T^{n-2}
\end{split}
\end{equation*}
Next, we use \eqnref{eqn gen MKmn2} to continue 
\begin{equation*}\label{}
\begin{split} 
 &= (n-1)^{m-1} P \cdot T^{n-1}+ 
  P \cdot T^{n-2}  \sum_{j=2}^m \sum_{S \subset V_2, \, \,  |S| \geq j} 
 (n-1)^{m-j} \Big( \sum_{i \in I_S } k'_i \Big)\\
&= (n-1)^{m-1} P \cdot T^{n-1}+ 
  P \cdot T^{n-2} \sum_{j=2}^m (n-1)^{m-j}  \sum_{S \subset V_2, \, \,  |S| = j} 
 \Big( \sum_{i \in I_S } k'_i \Big)\\
&= (n-1)^{m-1} P \cdot T^{n-1}+ 
  P \cdot T^{n-2} \sum_{j=2}^m (n-1)^{m-j} \binom{m-1}{j-1} T, \quad \text{by \lemref{lem counting}}\\
&= (n-1)^{m-1} P \cdot T^{n-1}+ 
  P \cdot T^{n-1} \sum_{j=1}^{m-1} (n-1)^{m-1-j} \binom{m-1}{j}, \quad \text{by the change of index}\\
&= P \cdot T^{n-1} \sum_{j=0}^{m-1} (n-1)^{m-1-j} \binom{m-1}{j}\\
&= n^{m-1}P \cdot T^{n-1}, \quad \text{by the Binomial Theorem}. 
\end{split}
\end{equation*} 
This completes the proof.
\end{proof}
In fact, \thmref{thm MKmn} and \corref{cor Kmn} are special cases of \thmref{thm gen MKmn}.

\section{Generalized Half Cone of $M^{k_1,k_2,\ldots, k_m} K_{m,n}$}\label{sec half cone bipartite}

In this section, we consider the graph
$F^k M^{k_1,k_2,\ldots, k_m} K_{m,n}$,  the generalized half cone of the modified complete bipartite graph $M^{k_1,k_2,\ldots, k_m} K_{m,n}$. The graph $F^k M^{k_1,k_2,\ldots, k_m} K_{m,n}$ has vertex set $V_1\cup V_2\cup \{ p\}$ as defined in \secref{sec introduction}, and it has $mk+k_1+k_2+\cdots+k_m$ edges.

The following preliminary lemma is needed for our computations below:
\begin{lemma}\label{lem counting2}
Let $A=\{ k_1, \, k_2, \, \ldots, \, k_m  \}$. For any integer $1 \leq j \leq m$, we have
\begin{equation*}
\begin{split}
\sum_{B \subset A, \, \,  |B| = j} \Big( \prod_{k' \in A-B } k' \Big) \Big( \sum_{k \in B } k \Big)  
=(m-j+1)\sum_{C \subset A, \, \,  |C| = m-j+1} \, \prod_{k \in C } k.
\end{split}
\end{equation*}
\end{lemma}
\begin{proof}
In this summation, there are $j \binom{m}{j}$ monomials such that each monomial is a product of $m-j+1$ elements from $A$. On the other hand, there can be at most $\binom{m}{m-j+1}$ different such monomials. That means, each possible such monomial appears $\frac{j \binom{m}{j}}{\binom{m}{m-j+1}}=m-j+1$ number of times in this summation.
\end{proof}

Next, we compute the number of spanning trees of the graph $F^k M^{k_1,k_2,\ldots, k_m} K_{m,n}$:
\begin{theorem}\label{thm gen half cone MKmn}
For any integers $k \geq 1$, $m \geq 1$, $n \geq 1$ and $k_i \geq 1$ for each $i= 1, \, 2, \ldots, m$, we have
\begin{equation*}\label{}
\begin{split}
t(F^kM^{k_1,k_2,\ldots, k_m} K_{m,n})= (k_1+k_2+\cdots+k_m)^{n-1} k \Big( \prod_{i=1}^m (k+k_i n) \Big) \sum_{i=1}^m \frac{k_i}{k+k_i n}.
\end{split}
\end{equation*}
In Particular, if $k_i=s$ for each $i=1, \, 2, \, \ldots, \, m$, then 
\begin{equation*}\label{}
\begin{split}
t(F^kM^{s,s,\ldots, s} K_{m,n})=  s^n m^n k (k+s n)^{m-1}.
\end{split}
\end{equation*}
\end{theorem}
\begin{proof}
Let $H=F^kM^{k_1,k_2,\ldots, k_m} K_{m,n}-p$. Note that $H$ is nothing but $M^{k_1,k_2,\ldots, k_m} K_{m,n}$.
Thus, by \thmref{thm gen MKmn},
\begin{equation}\label{eqn half gen MKmn1}
\begin{split}
t(H)= n^{m-1} k_1k_2\cdots k_m (k_1+k_2+\cdots+k_m)^{n-1}.
\end{split}
\end{equation}
Again, we set $T:=k_1+k_2+\cdots+k_m$ and $P:= k_1 k_2 \cdots k_m$.
We continue as in the proof of \thmref{thm gen MKmn}. Namely,
for $S= \{ q'_{1}, \, q'_{2}, \, \ldots, \, q'_{j} \} \subset V_2=\{ q_1, \, \ldots, \, q_m \}$, we have the complement set $V_2-S= \{ q'_{j+1}, \, q'_{j+2}, \, \ldots, \, q'_{m} \}$. Let $H_S$ be the graph obtained from $H$ by identifying the vertices in $S$. Suppose that the number of edges between the vertices 
$q'_i$ and $p_j$ is $k'_i$. Then one notes that 
$H_S$ is the same as the graph $M^{k'_{1}+ k'_{2}+ \cdots + k'_{j}, \, k'_{j+1},k'_{j+2},\ldots, k'_m} K_{m-j+1,n}$ with possibly different vertex labeling.
Thus, again by \thmref{thm gen MKmn},
\begin{equation}\label{eqn half gen MKmn2}
\begin{split}
t(H_S)= n^{m-j} (k'_{1}+ k'_{2}+ \cdots + k'_{j} ) k'_{j+1} k'_{j+2} \cdots k'_m (k_1+k_2+\cdots+k_m)^{n-1}.
\end{split}
\end{equation}
Since $N_{F^kM^{k_1,k_2,\ldots, k_m} K_{m,n}}(p)=V_2$, \thmref{thm main} gives
\begin{equation*}\label{}
\begin{split}
&t(F^kM^{k_1,k_2,\ldots, k_m} K_{m,n}) = mk \cdot t(H)+ \sum_{S \subset V_2, \, \,  |S| \geq 2} 
k^{|S|} t(H_{S})\\
&= mk \cdot t(H)+ \sum_{j=2}^m \sum_{S \subset V_2, \, \,  |S| = j} 
k^{j} t(H_{S})\\
&=mk \cdot t(H)+
  \sum_{j=2}^m  \sum_{S \subset V_2, \, \,  |S| = j} k^{j}
 n^{m-j} \Big( \prod_{i \in I_{V_2-S} } k'_i \Big) \Big( \sum_{i \in I_S } k'_i \Big)  T^{n-1}, \quad \text{ by \eqnref{eqn half gen MKmn2}}\\
&=mk \cdot t(H)+
 T^{n-1} \sum_{j=2}^m  k^{j}
 n^{m-j} \sum_{S \subset V_2, \, \,  |S| = j}  \Big( \prod_{i \in I_{V_2-S} } k'_i \Big) \Big( \sum_{i \in I_S } k'_i \Big)  \\
&=mk \cdot t(H)+
 T^{n-1} \sum_{j=2}^m  k^{j}
 n^{m-j} (m-j+1) \sum_{S \subset V_2, \, \,  |S| = m-j+1}  \Big( \prod_{i \in I_{S} } k_i \Big),
 \quad \text{by \lemref{lem counting2}} \\
\end{split}
\end{equation*}
We continue by using \eqnref{eqn half gen MKmn1},
\begin{equation*}\label{}
\begin{split}
&=mk n^{m-1} P T^{n-1} + T^{n-1} \sum_{j=2}^m  k^{j}
 n^{m-j} (m-j+1) \sum_{S \subset V_2, \, \,  |S| = m-j+1}  \Big( \prod_{i \in I_{S} } k_i \Big)\\
&=T^{n-1} \sum_{j=1}^m  k^{j}
 n^{m-j} (m-j+1) \sum_{S \subset V_2, \, \,  |S| = m-j+1}  \Big( \prod_{i \in I_{S} } k_i \Big).
\end{split}
\end{equation*}
Finally, we note that the sum
$$
\sum_{j=1}^m  k^{j}
 n^{m-j} (m-j+1) \sum_{S \subset V_2, \, \,  |S| = m-j+1}  \Big( \prod_{i \in I_{S} } k_i \Big)
 $$
is nothing but
$$
\Big[x \frac{d}{dy} \prod_{i=1}^{m} (x+kiy)  \Big]\Big|_{x=k,\, y=n}.
$$
Thus, the formula in the theorem follows.
\end{proof}

%
%
%
%
%
%
%
%

\textbf{Declaration of competing interest:} The author declares that he has no known competing financial interest or personal relationship that could have appeared to influence the work reported in this paper.


\begin{thebibliography}{999}

\bibitem[1]{Au} T. Austin, The enumeration of point labelled chromatic graphs and tress, { \em Canad. J. Math.}, 12 (1960), 535--545.

\bibitem[2]{Ca} A. Cayley, A theorem on trees, {\em Quart. J. Pure Appl. Math. }, 23 (1889), 376--378.




\bibitem[3]{C1} Z. Cinkir, Explicit Rayleigh's principles for resistive electrical network and the total number of spanning trees of graphs. Can be found at https://arxiv.org/abs/2411.02111.



\bibitem[4]{ER} O. Egecioglu and J. B. Remmel, Bijections for Cayley trees, spanning trees, and their q- analogues,  {\em  Journal of Combinatorial Theory, Series A}, 42 (1986), 15--30.




\bibitem[5]{HB} M. H. S. Haghighi and K. H. Bibak, The number of spanning trees in some classes of graphs, {\em Rocky Mountain Journal of Mathematics}, Vol 42, No 4, (2012), 1183--1195.


\bibitem[6]{KDT} K. M. Koh, F. M. Dong and E. G. Tay, Graphs and Their Applications, {\em Mathematical Medley},
Vol 32, No 2, (2005), 10--18.

\bibitem[7]{L} R. P. Lewis, The number of spanning trees of a complete multipartite graph,  { \em Discrete Mathematics} 197/198 (1999), 537--541.





\end{thebibliography}
\end{document}